\documentclass{article}
\usepackage{amssymb,amsfonts,amsmath,amsthm,amscd}
\usepackage{makeidx}
\usepackage[english]{babel}
\makeindex
\usepackage[x11names]{xcolor}
\theoremstyle{thm}
\newtheorem{thm}{Theorem}[section]

\theoremstyle{lem}
\newtheorem{lem}[thm]{Lemma}
\newtheorem{prop}[thm]{Proposition}
\newtheorem{defn}[thm]{Definition}
\theoremstyle{rem}
\newtheorem{rem}[thm]{Remark}
\newtheorem{exe}[thm]{Example}

\newcommand{\G}{\mathcal{G}}

\newcommand{\N}{\mathbb{N}}

\title{\textbf{Isomorphism Theorems for Groupoids and Some Applications}}
\author{Jes\'{u}s \'{A}vila and V\'{\i}ctor Mar\'{\i}n$^{\text{1}}$\\
   \small $^{\text{1}}$Departamento de Matem\'{a}ticas y Estad\'{i}stica \\
   \small Universidad del Tolima\\
   \small Santa Helena, Ibagu\'{e}, Colombia\\
   \small e-mail: javila@ut.edu.co, vemarinc@ut.edu.co\\
  H\'{e}ctor Pinedo $^{\text{2}}$\\
   \small $^{\text{2}}$ Escuela de Matem\'{a}ticas\\
   \small Universidad Industrial de Santander\\
   \small Cra. 27 calle 9, Bucaramanga, Colombia\\
   \small  e-mail:hpinedot@uis.edu.co}
   \date{\today}
\begin{document}
\maketitle
\begin{abstract}
\noindent
Using an algebraic point of view we present an  introduction to the groupoid theory, that is, we give fundamental properties of groupoids as, uniqueness of inverses and properties of the identities, and   study subgroupoids, wide subgroupoids and normal subgroupoids. We also present   the isomorphism theorems for groupoids and as applications, we obtain the corresponding version of the Zassenhaus Lemma and the Jordan-H\"{o}lder Theorem for groupoids. Finally inspired by  the Ehresmann-Schein-Nambooripad Theorem we improve a result of R. Exel concerning a one-to-one correspondence between partial actions of groups and actions of inverse semigroups.
\end{abstract}


\section{Introduction}
The concept of  groupoid from an algebraic point of view  appeared for the first time in \cite{B}. From this setting a (Brandt) groupoid can be seen as a generalization of a group, that is, a set with a partial multiplication on it that could contain many identities.

Brandt groupoids were generalized by C. Ehresmann in \cite{E}, where the author added further structures such as topological and differentiable.
Other equivalent definitions of groupoids and their properties appear in \cite{Br}, where a groupoid is defined as a small category where each morphism is invertible.

In \cite[Definition 1.1]{I} the author follows the definition given by Ehresmann  and presents the notion of groupoid as a particular case of universal algebra,  he defines strong homomorphism for groupoids and proves the correspondence theorem in this context. The Cayley Theorem for groupoids is also presented in \cite[Theorem 3.1]{I2}


Recently, some applications of groupoids to the study of partial actions are presented in different branches, for instance, in \cite{G} the author constructs a Birget-Rhodes expansion $\G^{\rm{BR}}$ associated with an ordered groupoid $\G$ and shows that  it  classifies partial actions of $\G$ on sets, in the topological context  in \cite{NY} is treated the globalization problem, connections between partial actions of groups and groupoids are given in \cite{A1, A2}.  Also,  ring theoretic and cohomological results of  global and partial actions of groupoids on algebras are obtained in
 \cite{Ba, BFP, BP, BPi2, NYOP, NYOP2, PF}. 
Galois theoretic  results for groupoid actions were obtained in \cite{BP, CT, PT1, PT}.  Finally globalization problem for partial groupoid actions have been considered in \cite{BPi, MP, NY}. 

In \cite{PT}  Paques and Tamusiunas give some structural definitions  in the context of groupoid such as abelian groupoid, subgroupoid, normal subgroupoid and show necessary and sufficient conditions for that a subgroupoid to be normal. Furthermore, they build quotient groupoids.

Due to the applications of the groupoids to partial actions, and their usefulness,  we will  give an elementary introduction to the  theory of groupoids from an axiomatic definition following Lawson  \cite{L}.

Our  principal goal in this work is to continue the algebraic development of  a groupoid theory. The paper is organized as follows.
After of introduction, in section 2, we present  groupoids from an  axiomatic point and  show some properties of them. In section 3 we recall the notions of some substructures of groupoids, such as subgroupoid, wide subgroupoid, and normal subgroupoid. In section 4  we  prove the  correspondence and isomorphism theorems for groupoids. In the final section we show an application of section four, we prove the Zassenhaus Lemma and the H\"{o}lder Theorem for groupoids; and we improve \cite[Theorem 4.2]{EX} using the Ehresmann-Schein-Nambooripad theorem.

It is important to note that the notion of groupoid can be presented from categories, algebraic structures, and universal algebra. In the last setting, the isomorphism theorems are valid, but the idea is to do an algebraic presentation and verify which assumptions are necessary. So it is possible to reach a wider audience.

\medskip

\section{Groupoids}

Now, we give two definitions of groupoids from an  algebraic point of view. 

\begin{defn}\cite[p. 78]{L}.\label{d2}
Let $\mathcal{G}$ be a set equipped with a partial binary operation on $\G\times \G$  which is denoted by concatenation. If $g,h \in \mathcal{G}$ and the product $gh$ is defined, we write $\exists gh$. An element $e\in \mathcal{G}$  is called an identity if
\begin{equation}\label{iden} \exists eg\,\,\,\text{ implies} \,\,\,eg=g\,\,\, \text{and}\, \,\,\, \exists g'e\,\,\,\text{ implies}\,\,\, g'e=g'.\end{equation}
 The set of identities of $\mathcal{G}$ is denoted by $\mathcal{G}_0$. Then $\mathcal{G}$ is said to be a groupoid if the following axioms hold:
\begin{enumerate}
\item[(i)] $\exists g(hl)$, if and only if, $\exists (gh)l$ and $g(hl)=(gh)l$.
\item[(ii)] $\exists g(hl)$, if and only if, $\exists gh$ and $\exists hl$.
\item[(iii)] For each $g\in \mathcal{G}$, there are unique identities $d(g)$ and $r(g)$ such that $\exists g d(g)$ and $\exists r(g)g.$

\item[(iv)] For each $g\in \mathcal{G}$, there is an element $g^{-1}\in \mathcal{G}$ such that $\exists g^{-1}g$, $\exists gg^{-1}$, $d(g)=g^{-1}g$, and $r(g)=gg^{-1}$.
\end{enumerate}
\end{defn}

The following definition of groupoid is presented in \cite[Definition 1.1]{Re}.
\begin{defn} \label{re}A groupoid is a  set endowed with a product map
\begin{equation}\label{prodg}\mathcal{G}^2\ni (g,h)\mapsto gh\in \G, \end{equation}
where the set $\mathcal{G}^2\subseteq \G\times \G$ is called the set of composible pairs and an inverse  map $\G\ni g \mapsto g^{-1}\in \G$ such that  for all $g,h,l\in \G$ the following relations are satisfied.
\begin{enumerate}
\item[(G1)] $(g^{-1})^{-1}=g;$
\item[(G2)] If $(g,h), (h,l)\in \G^2,$  then  $(gh,l), (g,hl)\in \G^2$ and $(gh)l=g(hl);$
\item[(G3)] $(g^{-1}, g)\in \G^2$  and  if $(g,h)\in \G^2,$ then  $g^{-1}(gh)=h;$

\item[(G4)]  $(g,g^{-1})\in \G^2$  and if $(l,g)\in \G^2,$ then  $(lg)g^{-1}=l.$
\end{enumerate}
\end{defn}

We shall check that Definitions \ref{d2} and \ref{re} are equivalent. First, we need a couple of lemmas.

\begin{lem}\label{l1}(\cite{F}, Lema 1.1.4) Suppose that $\G$ is a groupoid in the sense of Definition \ref{d2}.
Let $g,h \in \mathcal{G}.$ Then  $\exists gh$, if and only if, $d(g)=r(h)$.
\end{lem}

\begin{proof}
Let $g,h \in \mathcal{G}$ such that  $\exists gh$. By (iv) of Definition \ref{d2}, we have that $\exists g^{-1}, \exists h^{-1}$, $d(g)=g^{-1}g$, and $r(h)=hh^{-1}$. Since $ \exists gh$, then $\exists g^{-1}ghh^{-1}$. That is, $\exists d(g)r(h)$. Now, since $d(g)$ and $r(h)$ are identities, then $d(g)=d(g)r(h)=r(h)$. Conversely, if $d(g)=r(h)$, then $\exists gr(h)$ and since $gr(h)=g(hh^{-1})$ we have that $\exists g(hh^{-1})$. Whence by (ii) of Definition \ref{d2}, we have that $\exists gh$.
\end{proof}
\begin{lem} \label{unique}  Suppose that $\G$ is a groupoid in the sense of Definition \ref{d2}. Then the element $g^{-1}$ in (iv)  is unique and $(g^{-1})^{-1}=g$.
\end{lem}
\begin{proof}  For each $g\in \mathcal{G}$, assume that there exist $y,z\in \mathcal{G}$ such that $\exists yg$, $\exists gy$, $\exists zg$, $\exists gz$, $yg=d(g)=zg$, and $gy=r(g)=gz$. Notice that $yg=zg$ implies that $(yg)g^{-1}=(zg)g^{-1}$, which is defined by Definition \ref{d2} (ii), and then, by associativity, $yr(g)=zr(g)$. Thus $yd(y)=zd(z)$, and so $y=z$. Analogous for $gy=gz$. In particular, the inverse is unique.

Finally, the equality $(g^{-1})^{-1}=g$ follows from the uniqueness of the inverse of $g^{-1}.$
\end{proof}
We give the following.
\begin{prop} Let $\G$ be a set. Then it is a groupoid in the sense of Definition \ref{d2}, if and only if, it is a groupoid in the sense of Definition \ref{re}.
\end{prop}
\begin{proof} Let $\mathcal{G}^2=\{ (g,h)\in \mathcal{G}\times \mathcal{G}\mid \exists gh\}$. By using (iv) of Definition \ref{d2}, we define $\G\ni g \mapsto g^{-1}\in \G$. Then, by Lemma \ref{unique} this map is well defined.  We shall check  ($G1$)-($G4$) of Definition \ref{re}.

\noindent ($G1$) It is the second assumption in Lemma \ref{unique}.

\noindent ($G2$) If $(g,h), (h,l)\in \G^2,$ then $\exists gh$ and $\exists hl$. By (i) and (ii),  $\exists (gh)l$ and $\exists g(hl),$ that means $(gh,l), (g,hl)\in \G^2$ and $(gh)l=g(hl).$

\noindent ($G3$)  By item (iv),  we get that $(g^{-1}, g)\in\G^2 .$  Let $h\in \G$ with $(g,h)\in \G^2$.   By Lemma \ref{l1}, we get that $d(g)=r(h)$ and by using (iii)  we obtain $g^{-1}(gh)=h.$

\noindent ($G4$) This is proved analogously to the previous item.

Conversely, suppose that $\G$ is a set. We define a partial binary operation on $\G$ by $\exists gh,$ if and only if, $(g,h)\in \G^2$ and $gh=m(g,h).$ We shall check that  properties (i)-(iv) in Definition \ref{d2} hold.

\noindent (i) Let $g,h,l\in \G$ such that $\exists g(hl)$. Then $(g,hl), (h,l)\in \mathcal{G}^2$ and by ($G4$), $(l,l^{-1})\in \G^2$  and $h=(hl)l^{-1}.$ Thus, $(g,hl), (hl,l^{-1})\in \mathcal{G}^2 $ and by ($G2$), \linebreak $(g(hl),l^{-1})\in \mathcal{G}^2 $ and $[g(hl)]l^{-1}=g[(hl)l^{-1}]=gh$. In particular, $(g,h) \in \mathcal{G}^2 .$  We conclude that $(g,h), (h,l)\in \mathcal{G}^2$ and by using ($G2$), we  get that $(gh)l=g(hl).$ Conversely, suppose that $\exists (gh)l$. Then, $(g,h), (gh,l)\in \mathcal{G}^2$ and by ($G3$), we have that $(g^{-1}, g)\in \G^2$ and $g^{-1}(gh)=h$. Thus, $(g^{-1},gh), (gh,l)\in \mathcal{G}^2$ and by ($G2$), $(h,l)\in \G^2$. Finally, since $(g,h) \in \mathcal{G}^2 $ we obtain, again by ($G2$), that $(g,hl)\in \G^2$. Hence, $\exists g(hl).$

\noindent (ii) This is shown analogously to the previous items.

\noindent (iii) and (iv) If $g\in \G$, then $(g^{-1}, g), ( g, g^{-1})\in \G^2$. Thus, we set $r(g)=gg^{-1}$ and $d(g)=g^{-1}g$. Hence, by ($G2$), ($G3$) and  ($G4$), $\exists gd(g), \exists d(g)g$ and the equalities $gd(g)=g=r(g)(g)$ hold.
\end{proof}

\begin{rem} The interested reader can find another two equivalent definitions of groupoids in \cite{RA}  and \cite{We}.
\end{rem}
\begin{center}{\it From now on in this work $\G$ denotes a groupoid.}\end{center}

For the sake of completeness, we  give the proof of some known consequences of Definition \ref{d2}.

\begin{prop}\label{t1}(\cite{F}, Lema 1.1.4)
For each $g,h, k, l\in \G$ we have:
\begin{enumerate}

\item[(i)] If $\exists gh,$ then  $d(gh)=d(h)$ and $r(gh)=r(g)$.
\item[(ii)] $\exists gh$, if and only if, $\exists h^{-1}g^{-1}$ and in this case $(gh)^{-1}=h^{-1}g^{-1}$.

\end{enumerate}
\end{prop}

\begin{proof}

(i) For the first equality, we prove that $d(h)$ satisfy the axiom (iii) from Definition \ref{d2}. Indeed, assume that $\exists gh$. Then $\exists (gh)d(h)$, $d(g)=r(h)$  and
$$(gh)d(h)=g(hd(h))=gh.$$
In a similar way, it is possible to show that $r(gh)=r(g)$.\\

(ii) We have that $\exists gh$, if and only if, $d(g)=r(h)$. Notice that for any $l\in \G$ we have that  \begin{equation}\label{dgrg}
d(l)=l^{-1}l=l^{-1}(l^{-1})^{-1}=r(l^{-1}).\end{equation}

Then $d(h^{-1})=r(g^{-1})$. That is, $\exists h^{-1}g^{-1}$. Furthermore,

\begin{align*}gh(h^{-1}g^{-1})=g(hh^{-1})g^{-1}=gr(h)g^{-1}=gd(g)g^{-1}=gg^{-1}=r(g)=r(gh)\end{align*} and \begin{align*}(h^{-1}g^{-1})gh=h^{-1}(g^{-1}g)h=h^{-1}d(g)h=h^{-1}r(h)h=h^{-1}h=d(h)=d(gh).\end{align*}
Therefore, by the uniqueness of the inverse element we get that $(gh)^{-1}=h^{-1}g^{-1}$.
\end{proof}

The following statements also follow from the definition of groupoid.

\begin{prop}\label{propiedades de d y r}
Let $g\in \G$. Then the following statements hold:
\begin{flalign*}
 &(i)\,  d(g)=r(g^{-1}) & (ii)& \, d(d(g))=d(g) & (iii)& \, r(r(g))=r(g)\\
 &(iv)\,  d(r(g))=r(g) &  (v)& \, r(d(g))=d(g). &
\end{flalign*}
\end{prop}

\begin{proof}
  \begin{enumerate}
    \item[(i)] This is \eqref{dgrg}.
    \item[(ii)] $d(d(g))=d(g^{-1}g)=d(g),$ where the last equality follows from (i) of Proposition \ref{t1}.
    \item[(iii)] $r(r(g))=r(gg^{-1})=r(g),$  where the last equality  also follows from (i) of Proposition \ref{t1}.

  \end{enumerate}
Items (iv) and (v) are proved analogously. \end{proof}

\begin{rem} Let $\G$ be a  groupoid. In \cite[p. 3660]{BP}, Bagio and Paques called an element $e\in \mathcal{G}$ an $\textit{identity}$ if $e=d(g)$, for some $g\in \mathcal{G}$.

\end{rem}
\begin{prop}\label{equiv} Let $\G$ be a  groupoid. An element  $e$ of $\G$
is an identity in the sense of Bagio and Paques, if and only if, it satisfies \eqref{iden}.
\end{prop}

\begin{proof}
Suppose that $e=d(h)$ is an identity in the sense of Bagio and Paques, for some $h\in \mathcal{G}$. By (i) of Proposition  \ref{propiedades de d y r}, $e=d(h)=r(h^{-1}).$ Now, let $g,g'\in \G$ such that $\exists eg$ and $\exists g'e.$ By Lemma \ref{l1} and (ii) -- (v) of Proposition  \ref{propiedades de d y r}, we have that $d(g')=r(e)=e=d(e)=r(g)$, then
$eg=r(g)g=g$ and $g'e=g'd(g')=e$. Therefore, $e$ satisfies \eqref{iden}.

Conversely,  suppose that $e \in\mathcal{G}$ satisfies \eqref{iden}. By (iii) of Definition \ref{d2}, we get $\exists ed(e)$ and  $e=ed(e).$ Thus $e=d(e)$, and it follows that $e$ is an identity in the sense of Bagio and Paques.
\end{proof}

\begin{rem}\label{propiedades de las identidades} It follows from the proof of Proposition \ref{equiv}, that $d(e)=e=r(e)$,
  $\exists ee$, and $ee=e=e^{-1}$ for any  $e\in \G_0$. Moreover, note that the elements of $\mathcal{G}_0$ are the unique idempotents of $\mathcal{G}$. In fact, if $\exists g^2$ and $g^2=g$, then $(gg)g^{-1}=gg^{-1}$ and so $gr(g)=gg^{-1}=r(g)$. Since $d(g)=r(g)$, it follows that $g=gd(g)=gr(g)=r(g)$.
\end{rem}

\begin{prop}
  Let  $e\in \G_0$. Then, the set $\G_e=\{g\in G\mid d(g)=r(g)=e\}$ is a group.
\end{prop}
\begin{proof}
By Remark \ref{propiedades de las identidades}, we have that  $d(e)=r(e)=e.$ Thus $e\in \G_e.$  If $g,h\in \G_e$, then $d(g)=e=r(h)$, and so $\exists gh$ thanks to Lemma \ref{l1}. Now, (i) of Proposition \ref {t1}, implies that $d(gh)=d(h)=e$ and $r(gh)=r(g)=e$. Hence, $gh\in \G_e$. If $g\in \G_e$, then  by Lemma \ref{l1}, $\exists ge$ and $\exists eg$ and we have that   $ge=gd(g)=g$ and $eg=r(g)g=g$. Therefore, $e$ is the identity element of $\G_e$. Finally, let $g\in \G_e$. By Proposition \ref{propiedades de d y r}, $d(g^{-1})=r(g)=e$ and $r(g^{-1})=d(g)=e$. Hence, $g^{-1}\in \G_e$, $gg^{-1}=g^{-1}g=e,$ and we conclude that $\G_e$ is a group.
\end{proof}

\begin{defn}The group $\mathcal{G}_e$ is called the isotropy group associated to $e.$ The isotropy subgroupoid (see Definition \ref{subg}) or the group bundle associated to $\G$ is defined by the disjoint union ${\rm Iso}(\G)=\bigcup_{e\in \G_0}\G_e.$\end{defn}

\begin{rem}\label{abelian}
A concept of abelian groupoid was presented in \cite[p.111]{PT} as follows:
A groupoid  $\mathcal{G}$ is  abelian if $d(g)=r(g)$ for each $g\in \mathcal{G}$; and $gh=hg$ for all $g,h \in \mathcal{G}$ with $d(g)=r(h)$.
\end{rem}

We have the following.
\begin{prop}\label{abequiv} A groupoid $\G$ is abelian in the sense of Paques and Tamusiunas, if and only if, $\G={\rm Iso}(\G)$ and $\G_e$ is abelian for all $e\in \G_0.$
\end{prop}

In the light of Proposition \ref{abequiv}, we prefer to use the following definition of abelian groupoid.
\begin{defn}\label{ab}\cite[Definition 1.1]{MA} A groupoid $\G$ is called abelian if all its isotropy groups are abelian.
\end{defn}

Note that  if $\G$ is abelian in the sense of Paques and Tamusiunas, then it is abelian in the sense of  Definition \ref{ab}. Now, consider the groupoid $\G=\{g,g^{-1}, d(g), r(g)\}$ with $d(g)\neq r(g)$. Then, we have that $\mathcal{G}_{d(g)}=\{d(g)\}$ and $\mathcal{G}_{r(g)}=\{r(g)\}$. That is, $\mathcal{G}$ is an abelian groupoid in the sense of Definition \ref{ab}, but it is not a union of abelian groups.

\section{Normal subgroupoids, the quotient groupoid and homomorphisms}
In this section, we present a theory of substructures in a groupoid.
We follow the definition of subgroupoid given in \cite{PT}.

\begin{defn}\label{subg}
Let $\mathcal{G}$ be a groupoid and $\mathcal{H}$ a nonempty subset of $\mathcal{G}$.  $\mathcal{H}$ is said to be a subgroupoid of $\mathcal{G}$ if it satisfies: for all $g,h \in \mathcal{H}$,
\begin{enumerate}
\item[(i)] $g^{-1}\in \mathcal{H}$;
\item[(ii)] If $\exists gh$, then $gh\in \mathcal{H}$.
\end{enumerate}
If $\mathcal{H}$ is a subgroupoid of $\G,$ then it is called wide if $\mathcal{H}_0=\mathcal{G}_0$.

\end{defn}
\begin{rem} It is clear that if  $\mathcal{H}$ is a subgroupoid of $\G$, then it is a groupoid with the product  \eqref{prodg}, restricted to $\mathcal{H}^2=(\mathcal{H}\times \mathcal{H}) \cap \G^2.$
\end{rem}

\begin{exe} Let $\mathcal{G}$ be a groupoid.
\begin{enumerate}
  \item  Take $a\in \mathcal{G}$ such that $d(a)=r(a)$. The set $C(a)=\{g\in \mathcal{G}_{d(a)}\mid  ga=ag\}$ is a subgroupoid of $\mathcal{G}$. Indeed, first of all note that by assumption $a\in C(a)$. If $x,y\in C(a)$ then $x, y\in\mathcal{G}_{d(a)}$, $xa=ax$ and $ya=ay$. Since $\mathcal{G}_{d(a)}$ is a group,  then $\exists (xy)a$ and
$$(xy)a=x(ya)=x(ay)=(xa)y=(ax)y=a(xy).$$ That is, $xy\in C(a)$. If $x\in C(a)$, then $x^{-1}\in \mathcal{G}_{d(a)}$. Hence,  we have that $\exists x^{-1}a$ and $x^{-1}a=ax^{-1}$ since $ax=xa.$
 Observe that this example generalizes the concept of centralizer in groups.

  \item  Suppose that $\mathcal{G}$ is abelian  and $n>1$. Then the set $\mathcal{H}_n=\left\{a^n\mid a\in  {\rm Iso}(\G)\right\}$ is a subgroupoid of $\mathcal{G}$.
If $x,y\in \mathcal{H}_n$, then $x=a^n,y=b^n$ for some $a,b\in{\rm Iso}(\G)$. If $\exists xy$ then $\exists a^nb^n$, and this implies that $\exists ab$ and thus $a,b\in \G_e$ for some $e\in \G_0$. Then, $ab=ba$ and so $xy=(ab)^n\in \mathcal{H}$. Now, if $x\in \mathcal{H}$ then $x=a^n$ for some $a\in \mathcal{G}$. Thus, $x^{-1}=(a^n)^{-1}=(a^{-1})^n\in \mathcal{H}_n$. Finally, note that for $g\in \mathcal{G}$, $d(g)=d(g)^n\in \mathcal{H}_n$. Hence, $\mathcal{G}_0\subseteq \mathcal{H}$ and we conclude that $\mathcal{H}$ is wide.

  \item  Suppose that $\mathcal{G}$ is abelian. Then the set $Tor(\mathcal{G})=\{g\in  {\rm Iso}(\G)\mid g^n\in \mathcal{G}_0\textnormal{ for some }n\in \mathbb{N}\}$ is a wide subgroupoid of $\mathcal{G}$. First, it is clear that $\mathcal{G}_0\subseteq Tor(\mathcal{G})$. If $g,h\in Tor(\mathcal{G})$, then $g^n=e,h^k=f$ for some $n,k\in \mathbb{N}$ and some $e,f\in \mathcal{G}_0$. Thus, we obtain that $d(g)=d(g^n)=e=r(g^n)=r(g)$ and $d(h)=d(h^k)=f=r(h^k)=r(h)$. If $\exists gh$, then $d(g)=r(h)$ and thus $e=f$ and $gh=hg,$ since $\mathcal{G}$ is an abelian groupoid. Then $(gh)^{nk}=g^{nk}h^{nk}=(g^n)^k(h^k)^n=d(g)$, that is, $gh\in Tor(\mathcal{G})$. Now, since  $g^n=e$ we have $(g^{-1})^n=(g^n)^{-1}=e$ and hence $g^{-1}\in Tor(\mathcal{G})$. We conclude that $Tor(\mathcal{G})$ is a wide subgroupoid of $\mathcal{G}$. Note that if we take a fixed $n\in \mathbb{N}$ and define the set $\mathcal{D}_n=\{g\in \mathcal{G}\mid g^n\in \mathcal{G}_0\}$, then $\mathcal{D}_n$ is a wide subgroupoid of $\mathcal{G}$ and $\mathcal{D}_n\subseteq Tor(\mathcal{G})$. That is, $\mathcal{D}_n$ is a subgroupoid of $Tor(\mathcal{G})$. Observe that this example generalizes the concept of torsion subgroup in abelian groups.

\end{enumerate}
\end{exe}

\begin{prop}\label{producto HK=KH}
Let $\mathcal{G}$ be a groupoid and $\mathcal{H}$, $\mathcal{K}$  subgroupoids of $\mathcal{G}$.
Then:
\begin{enumerate}
\item [(i)] If $\mathcal{H}\mathcal{K}$ is non-empty, then $\mathcal{H}\mathcal{K}$ is a subgroupoid of $\mathcal{G}$,  if and only if, $\mathcal{H}\mathcal{K}=\mathcal{K}\mathcal{H}.$
\item [(ii)] If $\mathcal{H}$ and  $\mathcal{K}$ are wide and $\mathcal{H}\mathcal{K}$ is a subgroupoid,  then $\mathcal{H}\mathcal{K}$ is wide.
\end{enumerate}

\end{prop}
\begin{proof} The proof of (i) is similar to the group case. To prove  (ii), it is enough to observe that  $\G_0=\mathcal{H}_0=\mathcal{K}_0$ and if $e\in \G_0$, then $e=ee\in \mathcal{H}\mathcal{K}.$\end{proof}

Now, we present the notion of \textbf{normal subgroupoid} and prove several properties of them, which generalize well-known results in group theory. We follow the definition given in \cite{PT}.
\begin{defn}
Let $\mathcal{G}$ be a groupoid. The subgroupoid $\mathcal{H}$ of $\mathcal{G}$
is said to be normal, denoted by $\mathcal{H}\lhd \mathcal{G}$, if $g^{-1}\mathcal{H}g\neq \emptyset$ and $g^{-1}\mathcal{H}g\subseteq \mathcal{H},$ for all $g\in \mathcal{G}$. Where $g^{-1}\mathcal{H}g=\{g^{-1}hg\mid h\in \mathcal{H}\cap \G_{r(g)}\}.$
\end{defn}

\begin{rem}\label{carnor} By the proof of \cite[Lemma 3.1]{PT} one has that $g^{-1}\mathcal{H}g\neq \emptyset,$ if and only if, $\mathcal{H}$ is wide. Also the assertion $g^{-1}\mathcal{H}g\subseteq \mathcal{H}$ is equivalent to $g^{-1}\mathcal{H}_{r(g)}g=\mathcal{H}_{d(g)},$ for all $g\in \G.$
\end{rem}

Several examples of normal groupoids are presented in \cite[p.110 -111]{PT}.

Given  a wide subgroupoid  $\mathcal{H}$  of $\mathcal{G}$, in \cite{PT} Paques and Tamusiunas define a relation on $\mathcal{G}$ as follows:  for every $g,l \in \mathcal{G},$

$$ g\equiv_{\mathcal{H}}l \Longleftrightarrow (\exists l^{-1}g\,\,\,\,\text{ and} \,\,\,\,l^{-1}g\in \mathcal{H}).$$
 Furthermore, they prove that this relation is a congruence, which is an equivalence relation that is compatible with products. The equivalence class of $\equiv_{\mathcal{H}}$ containing $g$, is the set $g\mathcal{H}=\{gh\mid h\in \mathcal{H} \, \land \, r(h)=d(g)\}$. This set is called the left coset of $\mathcal{H}$ in $\mathcal{G}$ containing $g$. Then we have the following.

\begin{prop}\cite[Lemma 3.12]{PT}\label{quotient groupoid}
Let $\mathcal{H}$ be a normal subgroupoid of $\mathcal{G}$ and let $\mathcal{G/H}$ be the set of all left cosets of $\mathcal{H}$ in $\mathcal{G}$. Then $\mathcal{G/H}$ is a groupoid such that $\exists (g\mathcal{H})(l\mathcal{H})$, if and only if, $\exists gl$ and the partial binary operation is given by $(g\mathcal{H})(l\mathcal{H})=gl\mathcal{H}$.
\end{prop}

The groupoid $\mathcal{G/H}$ in Proposition \ref{quotient groupoid} is called the \textbf{quotient groupoid} of $\mathcal{G}$ by $\mathcal{H}$.

Now we present the notion of groupoid homomorphism and prove several properties of them, which generalize well-known results in homomorphisms of groups.

\begin{defn}
Let $\mathcal{G}$ and $\mathcal{G'}$ be groupoids. A map $\phi: \mathcal{G}\to \mathcal{G'}$ is called  groupoid homomorphism if for all $x,y\in \mathcal{G}$, $\exists xy$ implies that $\exists \phi(x)\phi(y)$, and in this case  $\phi(xy)=\phi(x)\phi(y)$.
\end{defn}

Notice that $j:\mathcal{G}\to \mathcal{G}/\mathcal{H}$ defined by $g\mapsto gH$ for all $g\in \mathcal{G}$, is a surjective groupoid homomorphism.

\begin{defn}
Let $\phi: \mathcal{G}\to \mathcal{G'}$ be a homomorphism of groupoids. We define the following sets:
\begin{enumerate}
\item[(i)] For $\mathcal{H}\subseteq \mathcal{G}$, write $\phi(\mathcal{H})=\{\phi(h) \in \mathcal{G'}\mid h\in \mathcal{H}\}$, the direct image of $\mathcal{H}.$ 
In particular, the set $\phi(\mathcal{G})$ is called the image of 
$\phi$.
\item[(ii)] $Ker(\phi)=\{g\in \mathcal{G}\mid \phi(g)\in \mathcal{G}'_0\}$, the kernel of $\phi$.
\item[(iii)] Let $\mathcal{H'}\subseteq \mathcal{G'}$, $\phi^{-1}(\mathcal{H'})=\{g\in \mathcal{G}\mid \phi(g)\in \mathcal{H'}\}$, the inverse image of $\mathcal{H'}$ by $\phi$.
\item[(iv)] $\phi$ is called a monomorphism if it  is  injective, an epimorphism if it is surjective, and an isomorphism if it is bijective.

\end{enumerate}
\end{defn}

\begin{rem}\label{subal} If $\G$ is abelian and $\mathcal{H}$ is a subgroupoid of $\G,$ then it is not difficult to show that $\mathcal{H}$ is abelian. Moreover, if $\G'$ is another groupoid, such that there is a groupoid epimorphism $\phi: \mathcal{G}\to \mathcal{G'},$ then $\G'$ is also abelian.
\end{rem}

\begin{prop}\label{phg}
Let $\phi: \mathcal{G}\to \mathcal{G'}$ be a groupoid homomorphism. Then:
\begin{enumerate}
\item[(i)] For each $a\in \mathcal{G}$, $\phi(d(a))=d(\phi(a))$, $\phi(r(a))=r(\phi(a))$ and $\phi(a^{-1})=(\phi(a))^{-1}$.
\item[(ii)] If $\mathcal{H'}$ is a subgroupoid of $\mathcal{G'}$, then $\phi^{-1}(\mathcal{H'})$ is a subgroupoid of $\mathcal{G}.$ Moreover, if $\mathcal{H'}$ is wide then $\phi^{-1}(\mathcal{H'})$ is wide, and it contains $Ker(\phi)$.
\item[(iii)] If $\mathcal{H'}\lhd \mathcal{G'}$, then $\phi^{-1}(\mathcal{H'})\lhd \mathcal{G}$ and  $Ker(\phi)\subseteq \phi^{-1}(\mathcal{H'})$. In particular, $Ker(\phi)\lhd \mathcal{G}$.

\end{enumerate}
\end{prop}

\begin{proof}(i) Let $a\in \mathcal{G}$. Since $\exists ad(a)$ then $\exists \phi(a)\phi(d(a))$ and $\phi(a)=\phi(ad(a))=\phi(a)\phi(d(a))$. Thus, by the uniqueness of the identities $\phi(d(a))=d(\phi(a))$. Analogously, $\phi(r(a))=r(\phi(a))$. Finally, since $\exists aa^{-1}$ and $\exists a^{-1}a$, then $\exists \phi(a)\phi(a^{-1})$ and $\exists \phi(a^{-1})\phi(a)$. Moreover, $$\phi(a)\phi(a^{-1})=\phi(aa^{-1})=\phi(r(a))=r(\phi(a))$$ and $$\phi(a^{-1})\phi(a)=\phi(a^{-1}a)=\phi(d(a))=d(\phi(a)).$$ Which implies that $\phi(a^{-1})=(\phi(a))^{-1}$.

(ii) It is not difficult to show that $\phi^{-1}(\mathcal{H'})$ is a  subgroupoid of $\mathcal{G}.$ Now suppose that $\mathcal{H'}$ is wide. By item (i), we know that $\phi (\mathcal{G}_0)\subseteq \mathcal{G}_0'\subseteq \mathcal{H}'$, that is, $\mathcal{G}_0\subseteq \phi ^{-1}(\mathcal{H}')$.
Finally, if $x\in Ker(\phi)$ then $\phi(x)\in \mathcal{G}_0'\subseteq \mathcal{H}'$ and hence $x\in \phi ^{-1}(\mathcal{H}'),$ as desired.

(iii) By item (ii), it is enough to see that $g^{-1}\phi^{-1}(\mathcal{H'})g\subseteq \phi^{-1}(\mathcal{H'})$ for all $g\in \mathcal{G}$. Indeed, let $g^{-1}lg\in g^{-1}\phi^{-1}(\mathcal{H'})g$ with $l\in \phi^{-1}(\mathcal{H'})$ and $d(l)=r(l)=r(g)$. Then, $\exists lg$ and thus $\phi(d(l))=\phi(r(g))$. We have, $$d(\phi(g^{-1})\phi(l))=d(\phi(l))=r(\phi(g)).$$ Then, $\exists (\phi(g^{-1})\phi(l))\phi(g)$ and since $\phi(l)\in \mathcal{H}'$ and $\mathcal{H'}\lhd \mathcal{G'}$ we obtain that,

 $$\phi(g^{-1}lg)= \phi(g^{-1})\phi(l)\phi(g)=
\phi(g)^{-1}\phi(l)\phi(g)\in \mathcal{H'}.$$

Finally, to show that $Ker(\phi)\lhd \mathcal{G}$, it is enough to observe that $\phi^{-1}(\mathcal{G'}_0)= Ker(\phi)$ and $\mathcal{G'}_0$ is  normal in $\G'$.
\end{proof}

\subsection{Strong isomorphism theorems for groupoids}
In this section, we present a special type of groupoid homomorphism, called $\mathbf {strong\,\, groupoid\,\, homomorphism}$. Using these  homomorphisms we show  the correspondence theorem and the isomorphism theorems for groupoids. This notion of  strong groupoid homomorphism has been considered before by several authors (see \cite[Remark 2.2]{I}).

\begin{defn}
Let $\phi: \mathcal{G}\to \mathcal{G'}$ be a groupoid homomorphism. $\phi$ is called  strong if for all $x,y\in \mathcal{G}$, $\exists \phi(x)\phi(y)$ implies that $\exists xy$.
\end{defn}

\begin{exe} Let $X$ be a nonempty set and $X^2=X\times X$. Then $X^2$ is a groupoid, where the product  is given by: $(y,z)(x,y)=(x,z),$ for $x,y,z\in X.$ Then, the map $f\colon \G \ni g\mapsto (d(g), r(g))\in  \G_0^2$ is a strong groupoid homomorphism with kernel ${\rm Iso}(\G).$
\end{exe}
\begin{prop}\label{png}
Let $\phi: \mathcal{G}\to \mathcal{G'}$ be a strong groupoid  homomorphism. Then:
\begin{enumerate}
\item[(i)] If $\mathcal{H}< \mathcal{G}$, then $\phi(\mathcal{H})<\mathcal{G}',$  
and $Ker(\phi)\mathcal{H}=\phi^{-1}(\phi(\mathcal{H}))$.
 In par\-ti\-cular, $Im(\phi)=\phi(\mathcal{G})$ and  $Ker(\phi)\mathcal{H}$ are  subgroupoids of $\mathcal{G'}$ and  $\mathcal{G},$ respectively.
\item[(ii)] If $\mathcal{H}\lhd \mathcal{G}$, then $\phi(\mathcal{H})\lhd \phi(\mathcal{G})$.
\item[(iii)](\cite{S}, Proposition 3.11) $\phi$ is an injective homomorphism, if and only if, $Ker(\phi)=\mathcal{G}_0$.
\item[(iv)] (The Correspondence Theorem for Groupoids) There exists a one-to-one correspondence between the sets $\mathfrak{A}=\{\mathcal{H}\mid \mathcal{H} < \mathcal{G}\, \land \, Ker(\phi)\subseteq \mathcal{H}\}$ and $\mathfrak{B}=\{\mathcal{H'}\mid \mathcal{H' }< \phi(\mathcal{G'})\}$. Moreover, this correspondence preserves normal subgroupoids.
\end{enumerate}
\end{prop}

\begin{proof} (i) It is clear that $\phi(\mathcal{H})\neq \emptyset$. Let $s,t\in \phi(\mathcal{H})$ and  suppose that $\exists st$. Then, $s=\phi(x),t=\phi(y)$ for some $x,y\in \mathcal{H}$. Since $\phi$ is strong, we have that $\exists xy$. Thus, $st=\phi(x)\phi(y)=\phi(xy)\in \phi(\mathcal{H})$. Now, if $y\in \phi(\mathcal{H})$ then $\phi(x)=y$ for some $x\in \mathcal{H}$, and we have $y^{-1}=\phi(x)^{-1}=\phi(x^{-1})\in \phi(\mathcal{H})$.

Now, we check the equality  $Ker(\phi)\mathcal{H}=\phi^{-1}(\phi(\mathcal{H})).$ If $g\in \phi^{-1}(\phi(\mathcal{H}))$, then there exists $h\in \mathcal{H}$ with $\phi(g)=\phi(h)$. Since $\phi$ is strong, we get that $\exists gh^{-1}$ and $gh^{-1}\in Ker (\phi)$. Hence, $g=(gh^{-1})h\in Ker (\phi)\mathcal{H}$. The other inclusion is clear.

(ii)-(iii) These are similar to the group case.

(iv) First, define the functions $\alpha :\mathfrak{A}\to \mathfrak{B}$ by $\alpha (\mathcal{H})=\phi(\mathcal{H})$ for each $\mathcal{H}\in \mathfrak{A}$, and $\beta:\mathfrak{B}\to \mathfrak{A}$ by $\beta (\mathcal{H}')=\phi^{-1}(\mathcal{H}')$ for each $\mathcal{H}'\in \mathfrak{B}$. By (i) of Proposition \ref{phg} and  (ii) of Proposition \ref{png}, it has that $\beta \circ \alpha=id_{\mathfrak{A}}$ and $\alpha \circ \beta=id_{\mathfrak{B}}$. That is, $\alpha$ is a bijective function. The remaining proof follows from the item $(iii)$ of Proposition \ref{phg} and  (ii)  above.
\end{proof}

Now we use strong homomorphisms to extend to the groupoid context a well-known result concerning the product of groups.
\begin{prop}\label{SIT}
Let $\mathcal{H}$ and $\mathcal{K}$ be subgroupoids of $\G$. If $\mathcal{K}$ is normal then:
\begin{enumerate}

  \item[(i)]   $\mathcal{H}\mathcal{K}$ is a subgroupoid of $\mathcal{G}$.
  \item[(ii)] If $\mathcal{H}$ is normal, then $\mathcal{H}\mathcal{K}$ is a normal subgroupoid of $\mathcal{G}$.
  \item[(iii)] If $\mathcal{H}$ is wide, then $\mathcal{H}\cap \mathcal{K}$ is a normal subgroupoid of $\mathcal{H}$.
\end{enumerate}
\end{prop}
\begin{proof}

    (i) Consider the groupoid epimorphism $\phi \colon \G\ni g\mapsto g\mathcal{K}\in \G/\mathcal{K}$. Then, by the definition of $ \G/\mathcal{K}$ the map  $\phi$ is strong and $Ker (\phi)=\mathcal{K}$. Thus, by (i) of Proposition \ref{png} we get that $\mathcal{K}\mathcal{H}=\phi^{-1}\phi(\mathcal{H})$ is a subgroupoid of $\G$. Hence, the result follows from Proposition \ref{producto HK=KH}.

    (ii) By the previous item, $\mathcal{H}\mathcal{K}$ is a subgroupoid of $\mathcal{G}$. Moreover, it is clear that  $\mathcal{H}\mathcal{K}$ is wide. Let $g\in \mathcal{G}$ and $x\in \mathcal{H}\mathcal{K}\cap \mathcal{G}_{r(g)}$. Then $x=hk$, with $h\in \mathcal{H}$, $k\in \mathcal{K}$, $d(h)=r(k)$ and $r(h)=d(k)=r(g)$. Thus $\exists g^{-1}hkg$ and we have that, $$g^{-1}hkg=g^{-1}hr(k)kg=g^{-1}hr(g)kg=(g^{-1}hg)(g^{-1}kg)\in \mathcal{H}\mathcal{K}.$$ That is, $\mathcal{H}\mathcal{K}$ is a normal subgroupoid of $\mathcal{G}$.

    (iii) It is clear that $\mathcal{H}\cap \mathcal{K}$ is a wide subgroupoid of $\mathcal{H}$. Let $g\in \mathcal{H}$ and $h\in \mathcal{H}\cap \mathcal{K}$ with $r(h)=d(h)=r(g)$. Then $\exists g^{-1}hg$ and by assumptions it follows that $g^{-1}hg\in \mathcal{H}\cap \mathcal{K}$.
\end{proof}

Next we present the  isomorphism theorems for groupoids.

\begin{thm}[The First Isomorphism Theorem]\label{FIT}
Let $\phi: \mathcal{G}\to \mathcal{G'}$ be a  surjective strong groupoid homomorphism. Then there exists a strong isomorphism $\overline{\phi}:\mathcal{G}/Ker(\phi)\to \mathcal{G'}$ such that $\phi=\overline{\phi}\circ j$, where $j$ is the canonical homomorphism of $\mathcal{G}$ onto $\mathcal{G}/Ker(\phi)$.
\end{thm}

\begin{proof}
Let $K=Ker(\phi)$. We define $\overline{\phi}:\mathcal{G}/K\to \mathcal{G'}$ as $\overline{\phi}(xK)=\phi(x)$, for each $xK\in \mathcal{G}/K$. First of all, we show that $\overline{\phi}$ is a well defined function. Indeed, assume that $xK=yK$. Then $\exists y^{-1}x$ and $y^{-1}x\in K$. That is $\phi(y^{-1}x)=d(l)$ for some $l\in \mathcal{G'}$, and then $\phi(y^{-1})\phi(x)=d(l)$. Since $\phi$ is surjective, then $l=\phi(z)$, for some $z\in \mathcal{G}$. Multiplying the above equation by $\phi(y)$, we have that

\begin{align*}
  \phi(y)\phi(y^{-1})\phi(x) & = \phi(y)d(l) \\
   & = \phi(y)d(\phi(z)) \\
   & = \phi(y)\phi(d(z)) \\
   & = \phi(yd(z)).
\end{align*}

Then $d(y)=d(z)$. So $\phi(yd(z))=\phi(yd(y))=\phi(y)$, whence $\phi(x)=\phi(y)$. Hence, $\overline{\phi}$ is well defined.

Now, note that $\overline{\phi}$ is a surjective strong homomorphism.  
Finally, we prove that $\overline{\phi}$ is injective. Indeed, assume that $\overline{\phi}(xK)=\overline{\phi}(yK)$, that is, $\phi(x)=\phi(y)$. Then, as $\phi$ is strong we have that $\phi(y^{-1}x)\in \mathcal{G'}_0$. Thus, $y^{-1}x\in K$ and we have that $xK=yK$.
\end{proof}

\begin{exe}1. Consider the identity function $i_{\mathcal{G}}$ of the groupoid $\mathcal{G}$. Then, it is clear that $i_{\mathcal{G}}$ is a surjective strong homomorphism and $Ker (i_{\mathcal{G}})=\mathcal{G}_0$. Thus, by the first isomorphism theorem, we obtain that $\mathcal{G}/\mathcal{G}_0\cong \mathcal{G}$.

2.  Consider the function $\theta:{\rm Iso}(\mathcal{G})\to {\rm Iso}(\mathcal{G})$, defined by $\theta (g)=d(g)$ for all $g\in {\rm Iso}(\mathcal{G})$. For $g,h\in {\rm Iso}(\mathcal{G})$ suppose that $\exists gh$. Then, $d(g)=r(h)=d(h)$,
 $\exists d(g)d(h)$, and 
$$\theta (gh)=d(gh)=d(h)=d(g)d(h)=\theta (g)\theta (h).$$
 Now, let $g,h\in {\rm Iso}(\mathcal{G})$ such that $\exists \theta (g)\theta (h)$. Then, $\exists d(g)d(h)$ which implies that $d(g)=d(h)$ and since $d(h)=r(h)$ we obtain $\exists gh$. In conclusion,  $\theta$ is a strong homomorphism, with $Ker (\theta)={\rm Iso}(\mathcal{G})$ and $Im (\theta )=\mathcal{G}_0$. Whence, by the first isomorphism theorem we obtain $ {\rm Iso}(\mathcal{G})/ {\rm Iso}(\mathcal{G})\cong \mathcal{G}_0$.

   3. Let $\mathcal{G}$ and $\mathcal{G}'$ be a groupoids. The set $\mathcal{G}\times \mathcal{G}'$ is a groupoid with the product defined by $\exists (x,y)\cdot (z,w)$ iff $\exists x\cdot z \wedge \exists y\cdot w$, and in this case $(x,y)\cdot (z,w)=(x\cdot z,y\cdot w)$. Moreover, note that $(\mathcal{G}\times \mathcal{G}')_0=\mathcal{G}_0\times \mathcal{G}'_0$. If $\mathcal{H}\lhd \mathcal{G}$ and $\mathcal{K}\lhd \mathcal{G}'$, then $\mathcal{H}\times \mathcal{K}\lhd \mathcal{G}\times \mathcal{G}'$ and $\mathcal{G}\times \mathcal{G}'/\mathcal{H}\times \mathcal{K}\cong (\mathcal{G}/\mathcal{H})\times (\mathcal{G}'/\mathcal{K})$.  Indeed, it is clear that $\mathcal{H}\times \mathcal{K}\lhd \mathcal{G}\times \mathcal{G}'.$
For the second affirmation, define $\psi: \mathcal{G}\times \mathcal{G}'\to (\mathcal{G}/\mathcal{H})\times (\mathcal{G}'/\mathcal{K})$ by $(g,g')\mapsto (g\mathcal{H},g'\mathcal{K})$, and note that $\psi$ is a strong homomorphism. Moreover,
  \begin{align*}
  Ker(\psi)&=\{(g,g')\in \mathcal{G}\times \mathcal{G}'\mid \psi(g,g')\in (\mathcal{G}/\mathcal{H})_0\times (\mathcal{G}'/\mathcal{K})_0\}\\
  &=\{(g,g')\in \mathcal{G}\times \mathcal{G}'\mid (g\mathcal{H},g'\mathcal{K})\in \mathcal{G}_0\mathcal{H}\times \mathcal{G'}_0\mathcal{K}\}\\
  &=\{(g,g')\in \mathcal{G}\times \mathcal{G}'\mid (g,g')\in \mathcal{H}\times \mathcal{K}\}\\
  &=\mathcal{H}\times \mathcal{K}.
  \end{align*}
  Thus, by the first isomorphism theorem the result follows.\\

\end{exe}

\begin{thm}\label{second}(The Second Isomorphism Theorem)
Let $\mathcal{G}$ be a groupoid, $\mathcal{M}$ a wide subgroupoid of $\mathcal{G}$ and $\mathcal{N}$ a normal subgroupoid of $\mathcal{G}$. Then, $\mathcal{M}\cap \mathcal{N}\lhd \mathcal{M}$  and
$$\frac{\mathcal{M}}{\mathcal{M}\cap \mathcal{N}}\cong \frac{\mathcal{M}\mathcal{N}}{\mathcal{N}}.$$
\end{thm}

\begin{proof}
First, note that by (i) of Proposition \ref{SIT}, $\mathcal{M}\mathcal{N}$ is a subgroupoid of $\mathcal{G}$. Moreover, since $\mathcal{N}\lhd \mathcal{G}$ we have $\mathcal{N}\lhd \mathcal{M}\mathcal{N}$. Also it is clear that, $\mathcal{M}\cap \mathcal{N}\lhd \mathcal{M}$.

We consider $\psi: \mathcal{M}\to \frac{\mathcal{M}\mathcal{N}}{\mathcal{N}}$ given by $\psi(m)=m\mathcal{N}$ for all $m\in \mathcal{M}$. Then, it is clear that  $\psi$ is a strong homomorphism. Furthermore, if $(mn)\mathcal{N}\in \frac{\mathcal{M}\mathcal{N}}{\mathcal{N}}$ then $(mn)\mathcal{N}=m\mathcal{N}$. Thus, $\psi$ is surjective. Now,
\begin{align*}
\left(\mathcal{M}\mathcal{N}/\mathcal{N}\right)_0&=\{d(x \mathcal{N})\mid x\mathcal{N}\in \mathcal{M}\mathcal{N}/\mathcal{N}\}\\
&=\{d(x)\mathcal{N}\mid x\in \mathcal{M}\mathcal{N}\}\\
&=\{d(x)\mathcal{N}\mid x=mn\, \land\, d(m)=r(n)\}\\
&=\{d(n)\mathcal{N}\mid n\in \mathcal{N}\}.
\end{align*}
On the other hand, $Ker(\psi)=\{m\in \mathcal{M}\mid m\mathcal{N}= d(n)\mathcal{N}\textnormal{ for some }n\in \mathcal{N}\}=\mathcal{M}\cap \mathcal{N}$. Indeed,  if $t\in \mathcal{M}\cap \mathcal{N}$ then $t\mathcal{N}=d(t^{-1})\mathcal{N}$ and thus $t\in Ker(\psi)$. For the other inclusion, if $m\in Ker(\psi)$ then $m\in \mathcal{M}$ and $m\mathcal{N}=d(n)\mathcal{N}$ for some $n\in \mathcal{N}$. Thus, $\exists d(n)m$ and $d(n)m\in \mathcal{N}$. That is, $d(n)=r(m)$ and we have $d(n)m=r(m)m=m\in \mathcal{N}$.
Finally, by Theorem \ref{FIT} we conclude that
$\frac{\mathcal{M}}{\mathcal{M}\cap \mathcal{N}}\cong \frac{\mathcal{M}\mathcal{N}}{\mathcal{N}},$ as  desired.
\end{proof}

\begin{rem} Given $M$ and $N$ as in  Theorem \ref{second},  we saw in the proof of the same theorem that  $\mathcal{N}\lhd \mathcal{M}\mathcal{N},$ which implies that $\mathcal{N}_0\subseteq \mathcal{M}_0.$ Indeed, let $e\in \mathcal{N}_0$. By Proposition \ref{equiv}, there is $n\in N$ such that $e=r(n)$. Since $n=m'n'$ with $(m',n')\in (M\times N)\cap \mathcal{G}^2,$ then $e=r(m')\in  \mathcal{M}_0. $ Conversely, the condition $\mathcal{N}_0\subseteq \mathcal{M}_0,$ clearly implies that $\mathcal{N}\lhd \mathcal{M}\mathcal{N}.$ From this, we conclude that for $\mathcal{M}$ and  $\mathcal{N}$ subgroupoids of $\mathcal{G},$ we have that $\mathcal{N}\lhd \mathcal{M}\mathcal{N},$ if and only if, $\mathcal{N}_0\subseteq \mathcal{M}_0.$
\end{rem}

\begin{thm}[The third Isomorphism Theorem]\label{TIT}
Let $\mathcal{G}$ be a groupoid, $\mathcal{H}\lhd \mathcal{G}$ and $\mathcal{K}\lhd \mathcal{G}$ with $\mathcal{K}\subseteq \mathcal{H}$. Then, $\mathcal{H}/\mathcal{K}\lhd \mathcal{G}/\mathcal{K}$ and
$$\frac{\mathcal{G}/\mathcal{K}}{\mathcal{H}/\mathcal{K}}\cong \mathcal{G}/\mathcal{H}.$$
\end{thm}

\begin{proof}
Define $\varphi: \mathcal{G}/\mathcal{K}\to \mathcal{G}/\mathcal{H}$ by $\varphi(g\mathcal{K})=g\mathcal{H}$. First of all, we show that $\varphi$ is a well defined function. Indeed, if $g\mathcal{K}=l\mathcal{K}$ then $\exists l^{-1}g$ and $l^{-1}g \in \mathcal{K}$. Since $\mathcal{K}<\mathcal{H}$, we have $l^{-1}g \in \mathcal{H}$ and hence $g\mathcal{H}=l\mathcal{H}$. Now,
\begin{align*}
Ker(\varphi)&=\{g\mathcal{K}\mid g\mathcal{H}\in (\mathcal{G}/\mathcal{H})_0\}=\{g\mathcal{K}\mid g\mathcal{H}\in \mathcal{G}_0\mathcal{H}\}=\{g\mathcal{K}\mid g\in \mathcal{H}\}=\mathcal{H}/\mathcal{K}.
\end{align*}
Thus, $\mathcal{H}/\mathcal{K}\lhd \mathcal{G}/\mathcal{K}$ and the Theorem \ref{FIT} implies that,
$$\frac{\mathcal{G}/\mathcal{K}}{\mathcal{H}/\mathcal{K}}\cong \mathcal{G}/\mathcal{H}.$$
\end{proof}

\section{Normal and subnormal series for groupoids}
In this section, we present some applications of the isomorphism theorems of groupoids to normal and subnormal series. In particular, we show that the Jordan-H\"{o}lder Theorem is also fulfilled in the context of  groupoids. First, we introduce the following natural definitions.

\begin{defn} Let $\G$ be a groupoid. Then:
\begin{itemize}
\item A subnormal series of a groupoid $\mathcal{G}$, is a chain of subgroupoids $\mathcal{G}^{(0)}=\mathcal{G}> \mathcal{G}^{(1)}>\cdots >\mathcal{G}^{(n)}$ such that $\mathcal{G}^{(i+1)}$ is normal in $\mathcal{G}^{(i)}$ for $0\leq i\leq n$. The factors of the series are the quotient groupoids $\mathcal{G}^{(i)}/\mathcal{G}^{(i+1)}$. The lenght of the series is the number of strict inclusions. A subnormal series such that $\mathcal{G}^{(i)}$ is normal in $\mathcal{G}$ for all $i$, is called  normal.

\item Let $\mathcal{S}:\mathcal{G}=\mathcal{G}^{(0)}>\mathcal{G}^{(1)}>\cdots >\mathcal{G}^{(n)}$ be a subnormal series. A one-step refinement of this series is any series of the form $\mathcal{G}=\mathcal{G}^{(0)}>\cdots >\mathcal{G}^{(i)}>N>\mathcal{G}^{(i+1)}>\cdots >\mathcal{G}^{(n)}$ or $\mathcal{G}=\mathcal{G}^{(0)}>\cdots > \mathcal{G}^{(n)}>N$, where $N$ is a normal subrgoupoid of $\mathcal{G}^{(i)}$ and $\mathcal{G}^{(i+1)}$ is normal in $N$ (if $i<n$ ). A refinement of a subnormal series $S$ is any subnormal series obtained from $S$ by a finite sequence of one-step refinements. A refinement of S  is called to be proper if it is larger than the length of S.

\item A subnormal series $\mathcal{G}=\mathcal{G}^{(0)}>\mathcal{G}^{(1)}>\cdots >\mathcal{G}^{(n)}=\mathcal{G}_0$ is a composition series if each factor $\mathcal{G}^{(i)}/\mathcal{G}^{(i+1)}$ is simple, that is its only normal subgroupoids are $\G$ and $\G_0,$ and it is
solvable if each factor is abelian.

\end{itemize}
\end{defn}

\begin{rem}  It follows from (iv) of Proposition \ref{png},  that
 if $\mathcal{N}$ is a normal subgroupoid of a groupoid $\mathcal{G}$, every normal subgroupoid of $\mathcal{G}/\mathcal{N}$ is of the form $\mathcal{H}/\mathcal{N}$ where $\mathcal{H}$ is a normal subroupoid of $\mathcal{G}$, which contains $\mathcal{N}$. Thus, if $\mathcal{G}\ne \mathcal{N}$ then $\mathcal{G}/\mathcal{N}$ is simple, if and only if, $\mathcal{N}$ is a maximal element in the set of all the normal subgroupoids $\mathcal{M}$ of $\mathcal{G}$, such that $\mathcal{M}\ne \mathcal{G}$.
\end{rem}

\begin{prop}\label{cs} Let $\G$ be a groupoid. Then:
\begin{enumerate}
\item[(i)]If $\mathcal{G}$ is finite, then it has a composition series.
\item[(ii)] Every refinement of a solvable series is a solvable series.
\item[(iii)] A subnormal series is a composition series, if and only if, it has no proper refinements.
\end{enumerate}
\end{prop}

\begin{proof}
(i) Let $\mathcal{G}^{(1)}$ be  a maximal normal subgroupoid of $\mathcal{G}$. Then, $\mathcal{G}/\mathcal{G}^{(1)}$ is simple by (iv) of Proposition \ref{png}. Let $\mathcal{G}^{(2)}$ be a maximal normal subgroupoid of $\mathcal{G}^{(1)}$, and so on. Now, since $\mathcal{G}$ is finite, this process must end with $\mathcal{G}^{(n)}=\mathcal{G}_0$. Thus, $\mathcal{G}=\mathcal{G}^{(0)}>\mathcal{G}^{(1)}>\cdots >\mathcal{G}^{(n)}=\mathcal{G}_0$ is a composition series.\\

(ii) Here we use Remark \ref{subal} to observe that if  $\mathcal{G}^{(i)}/\mathcal{G}^{(i+1)}$ is abelian and $\mathcal{G}^{(i+1)}\lhd \mathcal{H}\lhd \mathcal{G}^{(i)}$, then $\mathcal{H}/\mathcal{G}^{(i+1)}$ is abelian since it is a subgroupoid of $\mathcal{G}^{(i)}/\mathcal{G}^{(i+1)}$. Moreover, $\mathcal{G}^{(i)}/\mathcal{H}$ is abelian since it is isomorphic to $(\mathcal{G}^{(i)}/\mathcal{G}^{(i+1)})/(\mathcal{H}/\mathcal{G}^{(i+1)})$ by Theorem \ref{TIT}.\\

(iii) It follows from (iv) of Proposition \ref{png} and that a subnormal series $\mathcal{G}=\mathcal{G}^{(0)}>\mathcal{G}^{(1)}>\cdots >\mathcal{G}^{(n)}=\mathcal{G}_0$ has a proper refinement, if and only if, there  is a subgroupoid $\mathcal{H}$ such that for some $i, \mathcal{G}^{(i+1)}\lhd \mathcal{H}\lhd \mathcal{G}^{(i)}$ with $\mathcal{H}$ proper in $\mathcal{G}^{(i)}$ and $\mathcal{G}^{(i+1)}$ proper in $\mathcal{H}$.
\end{proof}

\begin{defn}
Two subnormal series $\mathcal{S}$ and $\mathcal{T}$ of a groupoid $\mathcal{G}$ are equivalent, if there is a one-to-one correspondence between the nontrivial factors of  $\mathcal{S}$ and the nontrivial factors of $\mathcal{T}$, such that the corresponding factors are isomorphic groupoids.
\end{defn}

\begin{lem}\label{BZ}
If $\mathcal{S}$ is a composition series of a groupoid $\mathcal{G}$, then any refinement of $\mathcal{S}$ is equivalent to $\mathcal{S}$.
\end{lem}

\begin{proof}
Let $\mathcal{S}: \mathcal{G}=\mathcal{G}^{(0)}>\mathcal{G}^{(1)}>\cdots >\mathcal{G}^{(n)}=\mathcal{G}_0$. By Proposition \ref{cs} (iii), $\mathcal{S}$ has no proper refinement. Thus, the only possible refinements of $\mathcal{S}$ are obtained by inserting additional copies of each $\mathcal{G}^{(i)}$. Whence, any refinement of $\mathcal{S}$ has exactly the same nontrivial factors as $\mathcal{S}$. Therefore, it is equivalent to $\mathcal{S}$.
\end{proof}

\begin{lem}[Zassenhaus Theorem for groupoids]\label{ZG}
Let $\mathcal{A}^*, \mathcal{A}, \mathcal{B}^*, \mathcal{B}$ be wide subgroupoids of a groupoid $\mathcal{G}$ such that:
\begin{itemize}
\item $\mathcal{A}^*$ is normal in $\mathcal{A},$
\item  $\mathcal{B}^*$ is normal in $\mathcal{B}.$

\end{itemize}
Then $\mathcal{A}^*(\mathcal{A}\cap \mathcal{B})$ and $\mathcal{B}^*(\mathcal{A}\cap \mathcal{B})$ are subgroupoids of $\G$ such that:
\begin{enumerate}
\item[(i)] $\mathcal{A}^*(\mathcal{A}\cap \mathcal{B}^*)$ is a normal subgroupoid of $\mathcal{A}^*(\mathcal{A}\cap \mathcal{B})$;
\item[(ii)] $\mathcal{B}^*(\mathcal{A}^*\cap \mathcal{B})$ is a normal subgroupoid of $\mathcal{B}^*(\mathcal{A}\cap \mathcal{B})$;
\item[(iii)] $\mathcal{A}^*(\mathcal{A}\cap \mathcal{B})/\mathcal{A}^*(\mathcal{A}\cap \mathcal{B}^*)\cong \mathcal{B}^*(\mathcal{A}\cap \mathcal{B})/\mathcal{B}^*(\mathcal{A}^*\cap \mathcal{B})$.
\end{enumerate}
\end{lem}

\begin{proof}
(i) Since $\mathcal{B}^*$ is normal in $\mathcal{B}$, $\mathcal{A} \cap \mathcal{B}^*=(\mathcal{A} \cap \mathcal{B})\cap \mathcal{B}^*$ is a normal subgroupoid of $\mathcal{A} \cap \mathcal{B}$  thanks to (iii)  of Proposition \ref{SIT}; similarly $\mathcal{A}^*\cap \mathcal{B}$ is normal in $\mathcal{A} \cap \mathcal{B}$. Then, $\mathcal{D}=(\mathcal{A}^*\cap \mathcal{B})(\mathcal{A} \cap \mathcal{B}^*)$ is a normal subgroupoid of $\mathcal{A} \cap \mathcal{B}$ by (ii) of Proposition \ref{SIT}. Also, by this same Proposition we have that $\mathcal{A}^*(\mathcal{A}\cap \mathcal{B})$ and $\mathcal{B}^*(\mathcal{A}\cap \mathcal{B})$ are subgroupoids of $\mathcal{A}$ and $\mathcal{B}$ respectively.
Now, we define
$$\tau: \mathcal{A}^*(\mathcal{A}\cap \mathcal{B})\ni ac\mapsto \tau(ac)=\mathcal{D}c\in(\mathcal{A}\cap \mathcal{B})/\mathcal{D},$$
 for all $a\in \mathcal{A}^*, c\in \mathcal{A}\cap \mathcal{B}$. The map $\tau$ is well defined since $ac=a_1c_1$ with $a,a_1\in \mathcal{A}^*; c,c_1\in \mathcal{A}\cap \mathcal{B}$, implies  
that,
$$c_1c^{-1}
=a_1^{-1}a\in (\mathcal{A}\cap \mathcal{B})\cap \mathcal{A}^*=\mathcal{A}^*\cap \mathcal{B}\subseteq \mathcal{D},$$ whence $\mathcal{D}c_1=\mathcal{D}c$. The map $\tau$ is clearly a strong epimorphism, %
and the equality $Ker(\tau)=\mathcal{A}^*(\mathcal{A} \cap \mathcal{B}^*)$ is shown in an analogous way to the group case. \\
Thus, Proposition \ref{phg} (iv),  implies that  $\mathcal{A}^*(\mathcal{A} \cap \mathcal{B}^*)$ is normal in $\mathcal{A}^*(\mathcal{A}\cap \mathcal{B})$ and by the first isomorphism theorem we get $\mathcal{A}^*(\mathcal{A} \cap \mathcal{B})/\mathcal{A}^*(\mathcal{A} \cap \mathcal{B}^*)\cong (\mathcal{A}\cap \mathcal{B})/\mathcal{D}$.\\
A symmetric argument shows that $\mathcal{B}^*(\mathcal{B} \cap \mathcal{A}^*)$ is normal in $\mathcal{B}^*(\mathcal{A}\cap \mathcal{B})$ and $\mathcal{B}^*(\mathcal{A} \cap \mathcal{B})/\mathcal{B}^*(\mathcal{A}^* \cap \mathcal{B})\cong (\mathcal{A}\cap \mathcal{B})/\mathcal{D}$. Whence (iii) follows.
\end{proof}

\begin{prop}[Schreier Theorem for groupoids]\label{SG}
Any two subnormal (resp. normal) series of a groupoid $\mathcal{G}$ have subnormal (resp. normal) refinement, which are equivalent.
\end{prop}

\begin{proof}
It follows from Lemma \ref{ZG}, (ii) of Proposition \ref{SIT}, and Proposition \ref{producto HK=KH} (1).
\end{proof}

\begin{prop}[Jordan-H\"{o}lder Theorem for groupoids]
Any two composition series of a groupoid $\mathcal{G}$ are equivalent.
\end{prop}

\begin{proof}
It follows from Proposition \ref{SG} and  Lemma \ref{BZ}.
\end{proof}

\subsection{Some remarks on the equivalence between inductive groupoids and inverse semigroups}

Recall that an inverse semigroup, is a semigroup $S$ such that for any $s\in S$ there is a unique $s^*\in S$ such that $s^*=s^*ss^*$ and $s=ss^*s.$
Now, let $X$ be a set and  consider the inverse semigroup $$\G(X)=\{f\colon A\to B\mid A\subseteq X, B\subseteq X\,\,\,\text{and $f$ is a bijection} \}.$$
 We recall the following.
\begin{defn} Let $S$ be a semigroup. An action of $S$ on $X$ is a semigroup homomorphism $\phi:S\to \G(X).$
\end{defn}
It follows from \cite[Theorem 4.2]{EX}, that partial actions of a group $G$ on $X$ are in one-to-one correspondence with  actions of $E(G)$ on $X,$ where $E(G)$ is the  semigroup  generated by the symbols $\{[g] \mid g \in G \}$ under
the following relations: For $g,h \in G$,
$$[g^{-1}][g][h] = [g^{-1}][gh],\,\,\, [g][h][h^{-1}] = [gh][h^{-1}],\,\,{\rm and}\,\, [g][1] = [g].$$ The semigroup $E(G)$ was introduced in \cite{EX}, and it is called the \emph{Exel semigroup of $G$}.
\begin{rem} \label{remE} Now we present some facts about $E(G).$
\begin{enumerate}
\item The semigroup $E(G)$ is a monoid with $1_{E(G)}=[1].$
\item  \cite[Proposition 2.5]{EX} For each $g\in G$ let $\gamma_g = [g][g^{-1}]$. Then, $\gamma_g$ is an idempotent of $E(G)$,
each element $\alpha \in E(G)$ may be uniquely written (up to the order of the $s_i$'s)
as
\begin{equation}\label{uniquely}\alpha = \gamma_{s_1} \gamma_{s_2} \ldots \gamma_{s_n} [g]\end{equation}
for some $s_1,s_2,\ldots,s_n, g\in G,$ with $ g\neq  s_i\neq s_j \neq g, i\neq j$ and $s_i\neq 1,$ for $i\in \{1,\cdots, n\}.$
From \eqref{uniquely}, follows that any idempotent in $E(G)$ has the form  $\gamma_{s_1} \gamma_{s_2} \ldots \gamma_{s_n} $ for some (uniquely) $s_1,s_2,\ldots,s_n, g\in G.$
\item \cite[Theorem 3.4]{EX} The set $E(G)$ is an inverse semigroup. In particular, the idempotents of $E(G)$ commute (see \cite[Theorem 3]{L}).
\end{enumerate}
\end{rem}

Given an inverse semigroup $S$ and $s,t\in S$, one defines  the {\it restricted product}
\begin{center}$s\cdot t$ exists if and only if  $s^*s=tt^*.$ \end{center} It follows from \cite[Proposition 3.1.4]{L}  and \cite[Proposition 4.1.1]{L}, that $(S, \cdot, \leq) $ is an inductive  groupoid (see \cite[p. 108]{L}), where $\leq$ is the natural partial order defined on $S.$ Then, by using the restricted product in $\G(X)$ we have that  $$\G^2(X)=\{(f,g)\in \G(X)\times \G(X)\mid {\rm im} g= {\rm dom} f\},$$
and $\G(X)$ is a groupoid with the product given by composition of maps restricted to $\G^2(X).$ Moreover, $\G(X)_0=\{{\rm id}_A\mid A\subseteq X\}.$

With respect to the semigroup $E(G)$ we have the following result.

\begin{prop}\label{idemp} Let $\alpha = \gamma_{s_1} \gamma_{s_2} \ldots \gamma_{s_n} [g]$ and $\beta= \gamma_{t_1} \gamma_{t_2} \ldots \gamma_{t_m} [l]$ where $s_1,s_2,\ldots,s_n, t_1,t_2,\ldots, t_m,g,l\in G$ are  as in \eqref{uniquely} of Remark \ref{remE}. Then $\alpha\alpha^*=\beta^*\beta,$ if and only if, $\{s_1,s_2, \cdots, s_n, g\}=\{t_1,t_2, \cdots, t_n, l^{-1}\}.$
\end{prop}
\begin{proof} We have that $\alpha\alpha^*=\gamma_{s_1} \gamma_{s_2} \ldots \gamma_{s_n} [g][g^{-1}]\gamma_{s_1} \gamma_{s_2} \ldots \gamma_{s_n} =\gamma_{s_1} \gamma_{s_2} \ldots \gamma_{s_n}\gamma_g.$
Then,
\begin{align*} \alpha\alpha^*=\beta^*\beta&\Longleftrightarrow \gamma_{s_1} \gamma_{s_2} \ldots \gamma_{s_n}\gamma_g=\gamma_{t_1} \gamma_{t_2} \ldots \gamma_{t_m} \gamma_{t^{-1}}\\
&\Longleftrightarrow \{s_1,s_2, \cdots, s_n, g\}=\{t_1,t_2, \cdots, t_n, l^{-1}\},
\end{align*} where the last equivalence follows from 2. of Remark \ref{remE}.\end{proof}
Using the restricted product to provide $E(G)$ with a groupoid structure we get by Lemma \ref{idemp} that,
$$E(G)^2=\{(\gamma_{s_1} \ldots \gamma_{s_n} [g],\gamma_{t_1} \ldots \gamma_{t_m} [l])\mid   \{s_1, \cdots, s_n, g\}=\{t_1, \cdots, t_n, l^{-1}\}\},$$ and $E(G)_{0}=\{\gamma_{s_1} \gamma_{s_2} \ldots \gamma_{s_n}\mid s_i\neq s_j, i\neq j, n\in \N\}.$

From \cite{NY}, follows that a  global action $\beta$ of a groupoid $\G$ on $X$ is a family of bijections $\beta=\{\beta_g\colon X_{g^{-1}}\to X_g\mid g\in \G\}$ such that:

\begin{itemize}
\item $X=\bigcup_{e\in \G_0} X_e;$
\item $\beta_e={\rm id}_{X_e},$ for all $e\in \G_0;$
\item $\beta_g\circ \beta_h=\beta_{gh},$ for all $(g,h)\in \G^2.$
\end{itemize}
Then  according to \cite[Proposition 10]{NY}, global actions of $\G$ on $X$  correspond to groupoid homomorphism $\G\to \G(X).$ On the other hand, in the case when  $\G$ is a group we obtain the definition of a partial group action on a set (see \cite[Definition 1.2]{E}) 

If $\mathcal{G}$ is an inductive groupoid, then  \cite[Proposition 4.1.7]{L} implies that $(\G,\otimes)$ is an inverse semigroup, where $\otimes$ denotes the pseudo product defined on $\G$ (see \cite[p. 112]{L}).

   Then we have the next.
\begin{prop} For every group G and any set X, there is a one-to-one correspondence between.
\begin{enumerate}
\item Partial actions of G on X.
\item Unital semigroup actions of $E(G)$ on $X.$
\item Groupoid homomorphisms $E(G)\to \G(X).$ 
\item Groupoid actions of $E(G)$ on $X.$
\end{enumerate}
\end{prop}
\begin{proof} Be have already observed that there is a one-to-one correspondence between
partial actions of $G$ on $X$ and semigroup actions of $E(G)$ on $X,$ and between groupoid homomorphisms $E(G)\to \G(X),$
and  global actions of $G$   on $X.$ Moreover, given a semigroup action $\phi\colon E(G) \to \G(X),$ then let $\beta_\alpha=\phi(\alpha)$ and $X_\alpha={\rm im}(\phi(\alpha)).$ Since $\phi([1])={\rm id}_X$, one has that the family $\beta=\{\beta_\alpha\colon X_{\alpha^{-1}}\to X_\alpha \}_{\alpha\in E(G)}$ is a global action of $E(G)$ on $X.$ Conversely, given a global action $\beta=\{\beta_\alpha\colon X_{\alpha^{-1}}\to X_\alpha \}_{\alpha\in E(G)}$ of $E(G)$ on $X,$ let $\varphi\colon E(G)\ni \alpha \mapsto \beta_\alpha\in  \G(X)$. Then, $\varphi$ is an  action of $E(G)$ on $X.$ Indeed, if $\alpha,\beta \in E(G)$, then by \cite[Proposition 4.1.7]{L} we have that $\alpha\gamma=\alpha\otimes \gamma$ and $\phi(\alpha\gamma)=\phi (\alpha\otimes \gamma)=\phi(\alpha)\otimes  \phi(\gamma)=\phi(\alpha)\circ  \phi(\gamma), $ as desired.\end{proof}

\subsection{Conflict of interest statement }
On behalf of all authors, the corresponding author states that there is no conflict of interest.

\end{document}